%% file: mcom-l-template.tex
\documentclass{amsart}
\usepackage{amssymb,cite}

\usepackage{graphicx}
\usepackage[usenames,dvipsnames]{xcolor}
\usepackage[utf8]{inputenc}
\usepackage{listings}
\usepackage{bbm,overpic}

\input{listing.tex} 

\newtheorem{theorem}{Theorem}[section]
\newtheorem{lemma}[theorem]{Lemma}
\newtheorem{proposition}{Proposition}[section]

\theoremstyle{definition}
\newtheorem{definition}[theorem]{Definition}
\newtheorem{notation}[theorem]{Notation}
\newtheorem{example}[theorem]{Example}

\theoremstyle{remark}
\newtheorem{remark}[theorem]{Remark}

\numberwithin{equation}{section}

\DeclareMathOperator{\D}{d}
\DeclareMathOperator{\E}{e}
\DeclareMathOperator{\tr}{tr}
\DeclareMathOperator{\I}{i}

\begin{document}

\title{On spectral and numerical properties of random butterfly matrices}


\author{Thomas Trogdon}
\address{Department of Mathematics, University of California, Irvine}
\email{ttrogdon@math.uci.edu}


\thanks{This work was supported in part by grant NSF DMS-1753185.}


\date{}

\dedicatory{}

\begin{abstract}
Spectral and numerical properties of classes of random orthogonal butterfly matrices, as introduced by Parker (1995), are discussed, including the uniformity of eigenvalue distributions. These matrices are important because the matrix-vector product with an $N$-dimensional vector can be performed in $O(N \log N)$ operations. And in the simplest situation, these random matrices coincide with Haar measure on a subgroup of the orthogonal group.  We discuss other implications in the context of randomized linear algebra. 
\end{abstract}

\maketitle


\bibliographystyle{amsplain}


\section{Introduction}

In this paper we discuss properties of several classes of random butterfly matrices.  Loosely speaking, butterfly matrices are matrices in $\mathbb R^{N\times N}$, $N = 2^n$, which are defined recursively.  The most commonly encountered example (in $\mathbb C^{N\times N}$) is the matrix representation for the discrete (or fast) Fourier transform \cite{VanLoan1992}.  Other examples include Hadamard matrices.  A definition is given in Section~\ref{s:butterfly}.  The primary utility of such matrices is that they can be be applied (matrix times a vector) in much less than $N^2$ operations --- typically $2 N n$ operations.  For two classes of random butterfly matrices we prove that the eigenvalues are distributed uniformly on the unit circle in $\mathbb C$.

Butterfly matrices, and in particular, random butterfly matrices, were introduced by Stott Parker \cite{Parker1995} as a mechanism to avoid pivoting in Gaussian elimination.  One randomizes the a linear system $Ax = b$ by applying a random orthogonal (or unitary) transformation to the columns, giving $\Omega Ax = \Omega b$.  After randomization, one can show that if $A$ is nonsingular then the upper left $k \times k$, $k = 1,2,\ldots$  blocks of $\Omega A$ are also nonsingular.  This implies that Gaussian elimination needs no pivoting to complete, i.e. $\Omega A$ has an $LU$-factorization.

\begin{figure}[htbp]
  \centering
  \begin{overpic}[width=.3\linewidth]{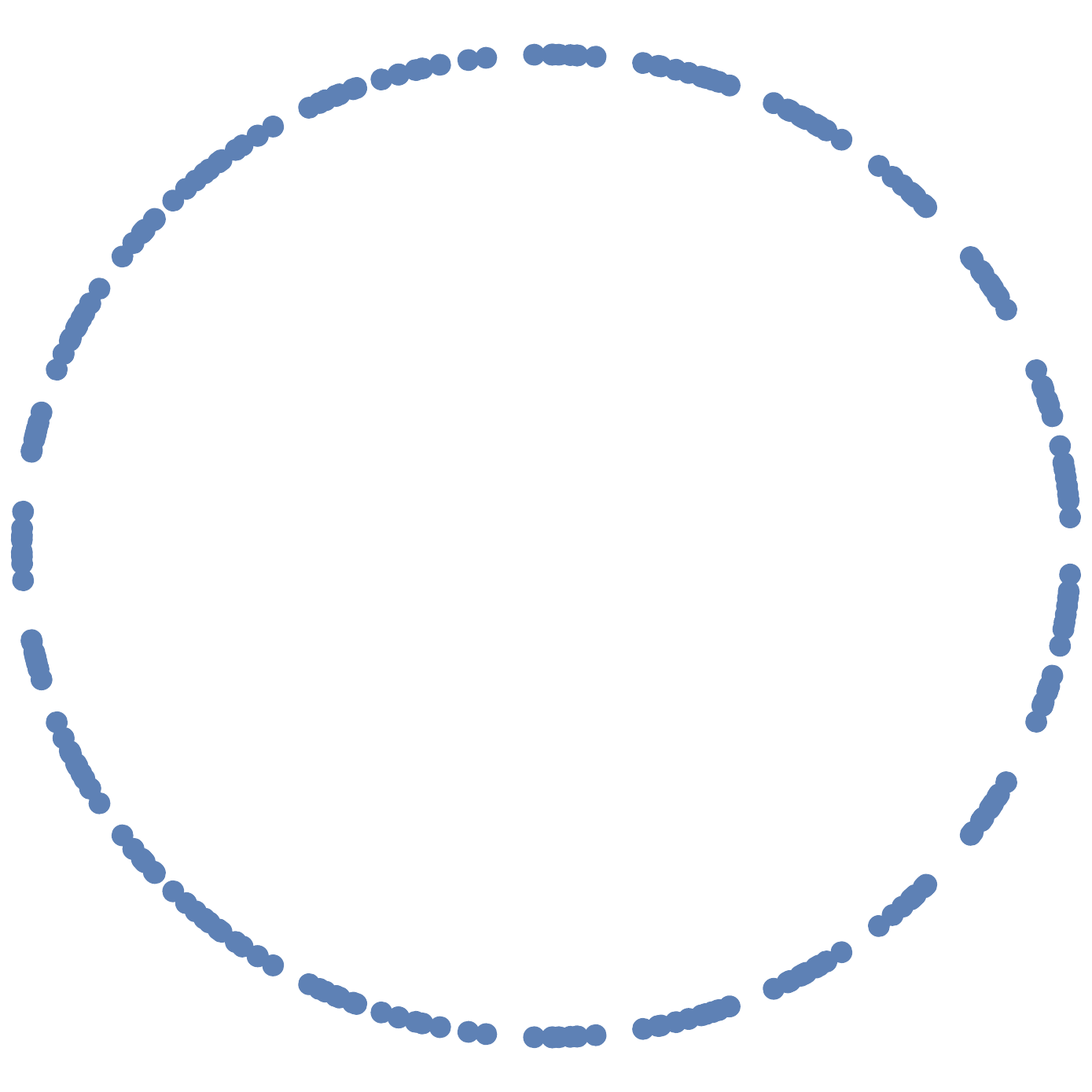}
    \put(5,5){(a)}
  \end{overpic}
  \begin{overpic}[width=.3\linewidth]{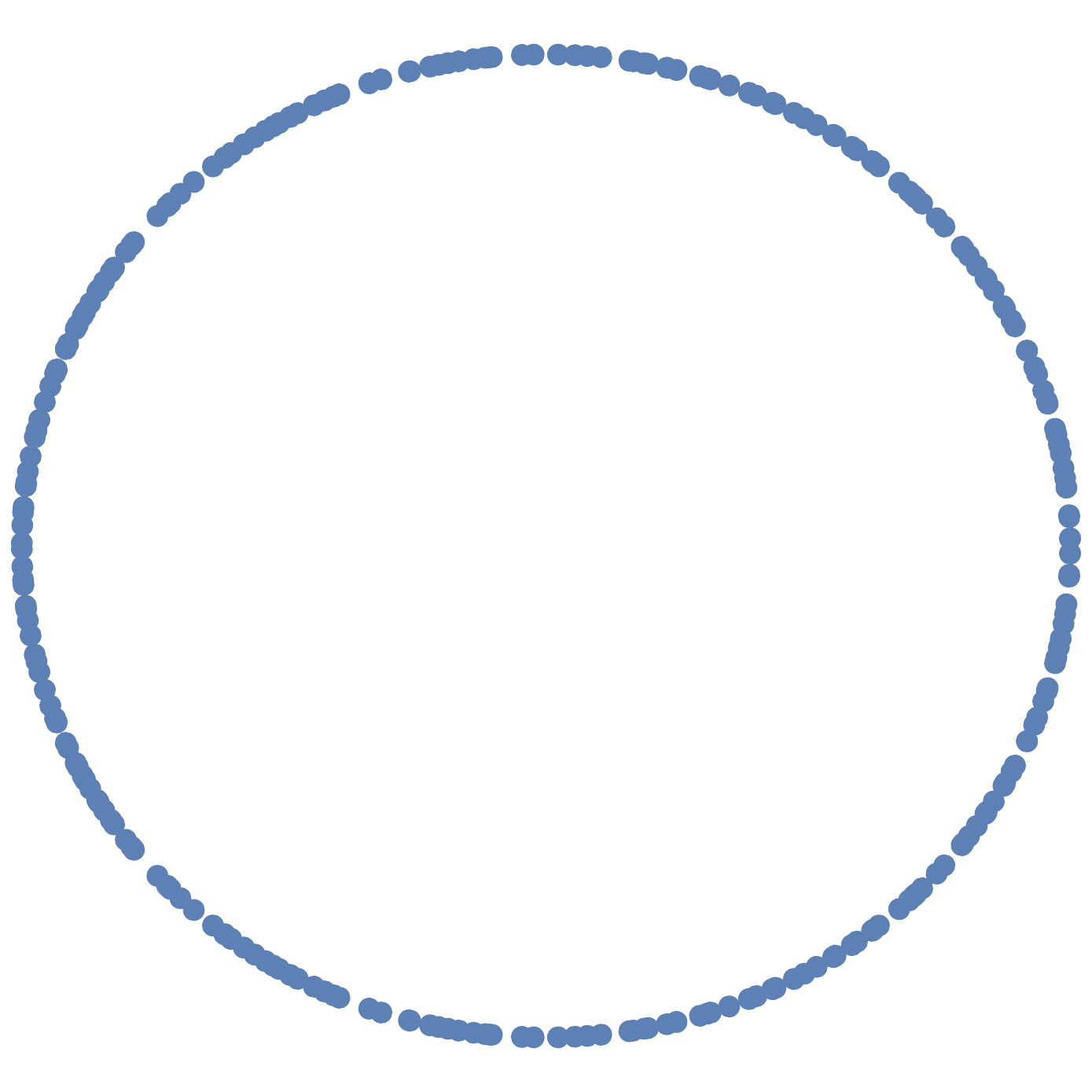}
    \put(5,5){(b)}
  \end{overpic}\\
  \begin{overpic}[width=.3\linewidth]{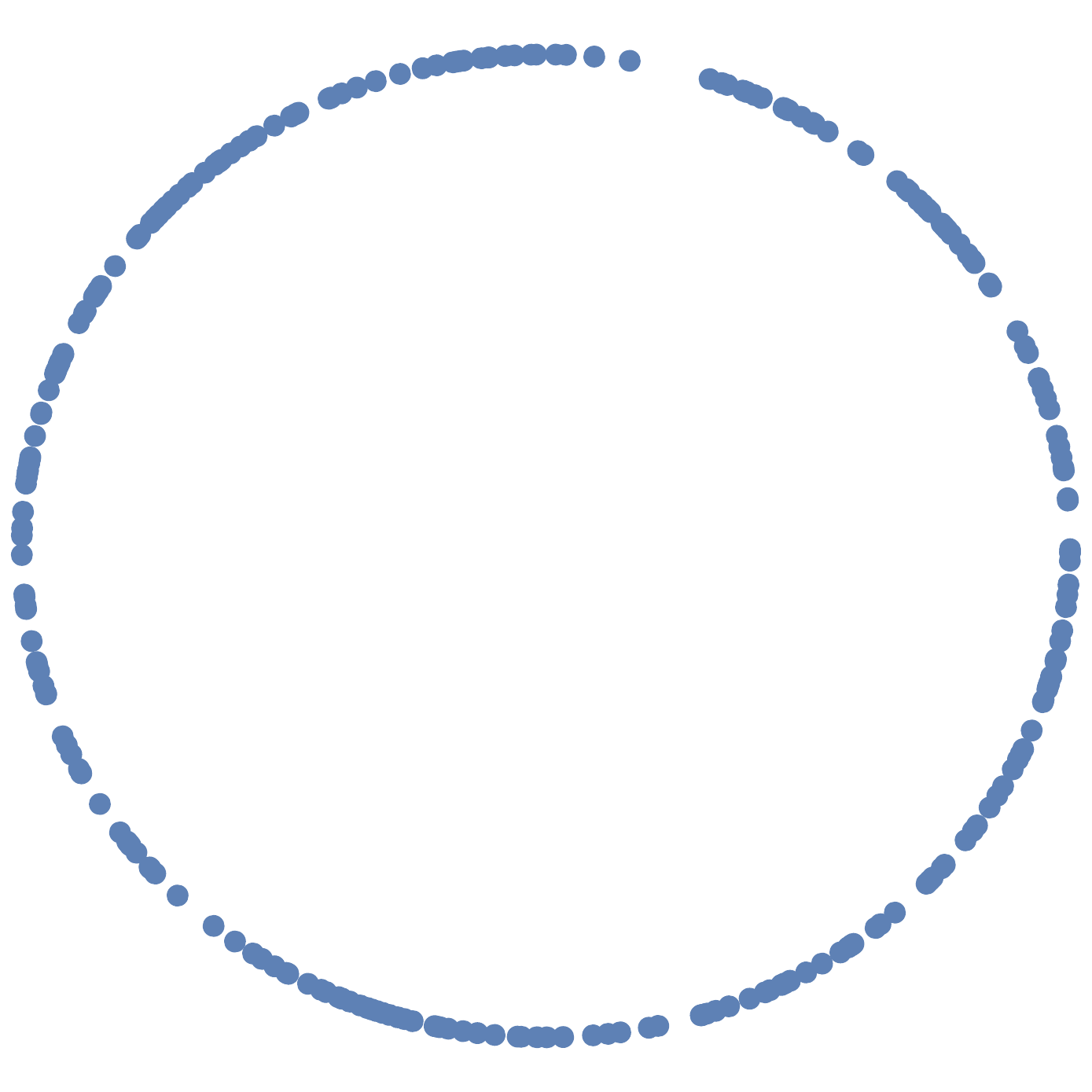}
    \put(5,5){(c)}
  \end{overpic}
  \begin{overpic}[width=.3\linewidth]{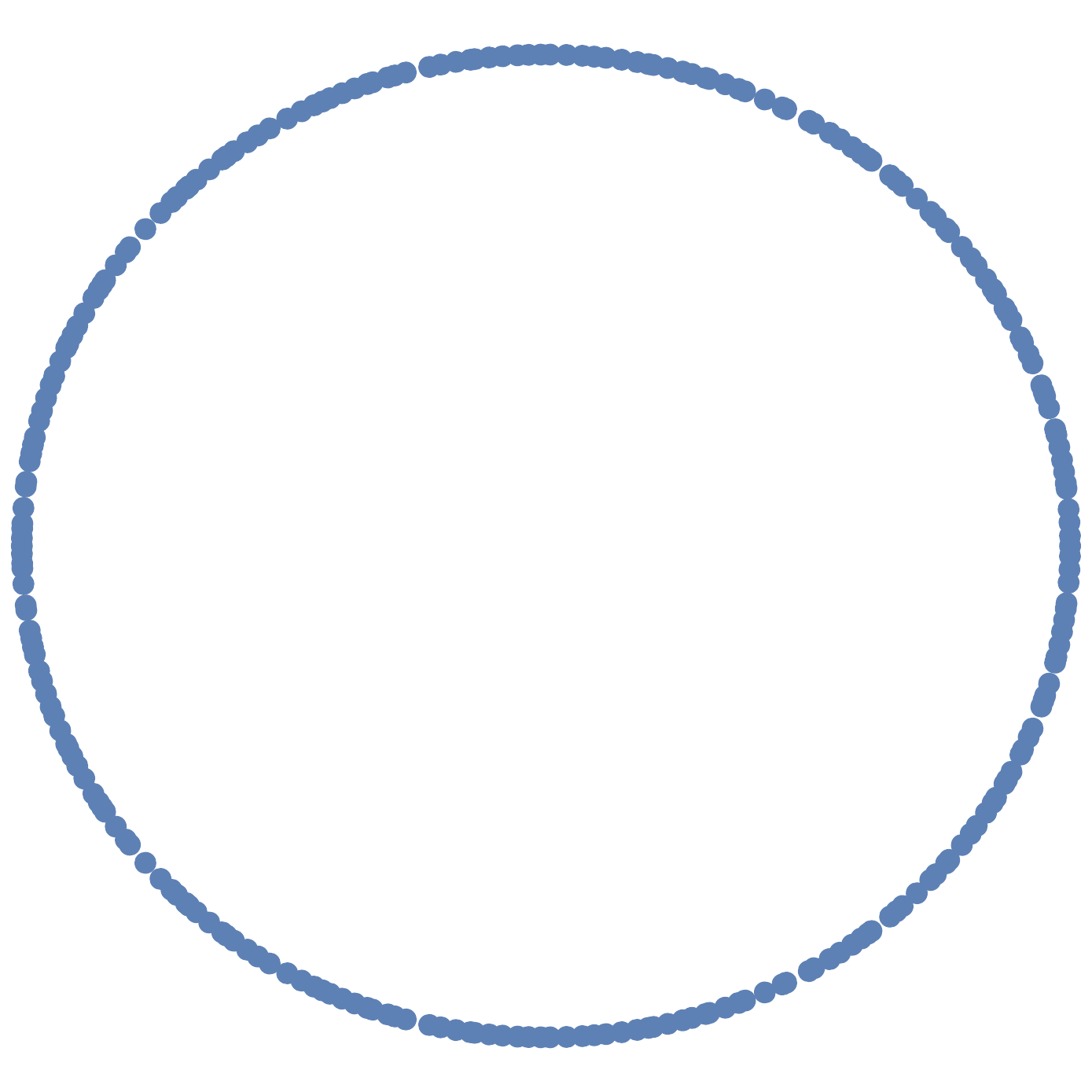}
    \put(5,5){(d)}
  \end{overpic}
  \caption{\label{f:typical} The eigenvalues of typical random orthogonal matrices of size $256 \times 256$ drawn from different distributions.  (a) The eigenvalues of Haar-butterfly matrices from Section~\ref{s:haarbutt}.  These eigenvalues tend to appear in groups.  (b) The eigenvalues of non-simple random butterfly matrices from Section~\ref{s:general}.  Here groups are less apparent.  (c) Independent and identically distributed points on $\{|z| = 1\}$.  Poisson clumping is evident.  (d) The eigenvalues of a typical Haar matrix on $\mathrm{O}(256)$.}
\end{figure}

Similar randomization has been employed in the context of least squares problems and low-rank approximations, see \cite{Avron2010,Halko2011,Liberty2007,Lerman2014} and the references therein.  The main idea for low-rank approximation is that for $A \in \mathbb C^{N \times M}$, $N > M$ one can often randomize $A$ to better approximate its range. For a rank $K$ approximation, one considers $A\Omega$ for random matrix $\Omega$ of dimension $M \times L$, $L = K + p$ for some oversampling\footnote{In practice, $p = 5,10$ work well \cite{Halko2011}.} $p > 0$.  Then a $QR$ factorization $A\Omega = QR$ can be found and $Q$ \underline{should} be a good rank $K$ approximation of the range of $A$.  The random matrix $\Omega$ is often taken to be a so-called \emph{subsampled randomized Fourier transform} (SRFT) matrix, which can be seen as a subclass of butterfly matrices (over $\mathbb C$, but we consider real matrices in this work).

We emphasize that random butterfly matrices (and SRFT matrices) are used for the sole reason that matrix-vector multiplication is fast.  To truly uniformize the matrix $A$ one might want to sample a matrix $Q$, at uniform, from Haar measure on the orthogonal (or unitary) group in $\mathbb R^{M \times M}$ (or $\mathbb C^{M \times M}$) and then subsample $L$ its columns\footnote{Uniformity tells us that it suffices to take the first $L$ columns.} to form $\Omega$. Such a matrix $Q$ is referred to as a Haar matrix. Then $A \Omega$ should be, in some sense, in it is most generic state.  But computing the first $L$ columns of an orthogonal matrix sampled uniformly at random requires $O(M^2 L)$ operations (see \cite{Stewart1980} and Section~\ref{s:haar} below), and then multiplying $A\Omega$ requires $O(NML)$ operations.  But, as is well-known, the recursive structure of butterfly matrices $B$ allow one to subsample and reduce the computation of $AB$ to purely $O(NM \log_2 L)$ operations \cite{Halko2011}.  See also \cite{Boutsidis2013,Tropp2011} and Section~\ref{s:code_haarrbm} for discussions of subsampling.    

In the context of solving a least squares system $\min_x \|Ax - b\|_2$, $A \in \mathbb C^{N \times M}$ and $N \gg M$ one may want to subsample the rows of $A$ to consider a reduced system with fewer equations (but the same number of unknowns).  The choice of which rows to discard \emph{a priori} is difficult.  If $A$ has repeated rows and only these rows are sampled, by chance, the subsampled matrix will fail to be of full rank.  One can randomize $A$ first by applying a random butterfly matrix to each of its columns and then subsample the rows.  This can perform well \cite{Avron2010}.

The main goal of this paper is to shed light on the interesting mathematical and numerical properties of the infrequently-discussed random butterfly matrices.  This is accomplished by (1) comparing the (statistical) spectral properties of random butterfly matrices to that of Haar matrices and matrices with independent and identically distributed eigenvalues and (2) comparing the performance of these matrices in the above linear algebra contexts to both Haar matrices and SRFT matrices.  In the remainder of this section we define (random) butterfly matrices and Haar matrices and introduce preliminaries from random matrix theory.  In Section~\ref{s:haarbutt} we establish properties of the simplest class of random butterfly matrices which are identified with Haar matrices on a subgroup of the special orthogonal group. In Section~\ref{s:general} we give an additional generalization.  Finally, in Section~\ref{s:numerics} we perform numerical experiments with random butterfly matrices to test their uniformization properties.  A visual comparison of the spectrum of a typical random butterfly matrix is given in Figure~\ref{f:typical}.

\begin{remark}
In this paper we focus on real random orthogonal matrices of size $N \times N$, $N = 2^n$.  Important generalizations to consider are extensions to complex unitary matrices and $N = 1,2,3,\ldots$.
\end{remark}

\subsection{Butterfly matrices}\label{s:butterfly}

Unlike the work of Parker \cite{Parker1995} we concentrate on orthogonal butterfly matrices.  We also use the term \emph{butterfly matrix} to refer the recursively-defined matrix itself.
\newcommand{\pp}[1]{^{(#1)}}
\begin{definition}
A recursive orthogonal butterfly matrix $B\pp{N} \in \mathrm{O}(N)$ (or just butterfly matrix) of size $N = 2^n$ is given by the $1\times 1$ matrix $\begin{bmatrix} 1 \end{bmatrix}$ when $n = 0$ and
\begin{align}\label{e:buttdef}
B \pp{N} = \begin{bmatrix} \phantom{.}~~C_{n-1} A_1\pp{N/2}  & S_{n-1} A_2\pp{N/2}  \\ - S_{n-1} A_1 \pp{N/2}  & C_{n-1} A_2\pp{N/2}  \end{bmatrix}, \quad n > 0.
\end{align} 
Here $A_1\pp{N/2}$ and $A_2\pp{N/2}$ are both butterfly matrices of size $N/2 \times N/2 = 2^{n-1} \times 2^{n-1}$ and $C_{n-1}$ and $S_{n-1}$ are symmetric $2^{n-1}\times 2^{n-1}$ matrices satisfying $C_{n-1}^2 + S_{n-1}^2 = I$ and $C_{n-1}S_{n-1} = S_{n-1}C_{n-1}$.
\end{definition}

\begin{example}
  If $C_{n-1}= c_{n-1}I$, $S_{n-1} = s_{n-1}I$ and $A_1\pp{N/2} = A_2\pp{N/2}$ then
  \begin{align}
    B\pp{N} = \begin{bmatrix} c_{n-1} & s_{n-1} \\ - s_{n-1} & c_{n-1} \end{bmatrix} \otimes A_1 \pp{N/2},
  \end{align}
  where $\otimes$ denotes the \emph{Kronecker} product.  We refrain from using Kronecker products because this notation is not convenient for the generality in Definition~\ref{e:buttdef}.
\end{example}  

\begin{example}
  A unitary example of a butterfly-like matrix is the matrix for the discrete Fourier transform, see \cite[Section 1.4]{VanLoan1992}.  Let $\omega = \E^{- \I \pi/2}$ be a 4th-root of unity.  Define
  \begin{align}
    F_4 = \begin{bmatrix} 1 & 1 & 1 & 1\\
      1 & \omega & \omega^2 & \omega^3\\
      1 & \omega^2 & \omega^4 & \omega^6\\
      1 & \omega^3 & \omega^6 & \omega^9 \end{bmatrix}, \quad F_2 = \begin{bmatrix} 1 & 1 \\  1 & \omega^2 \end{bmatrix}.
  \end{align}
  Note that $\omega^2$ is a square-root of unity and
  \begin{align}
    F_4 \begin{bmatrix} 1 & \\ && 1 \\ &1 \\ &&& 1 \end{bmatrix} = \begin{bmatrix} 1 & 1 & 1 & 1\\
      1 & \omega^2 & \omega & \omega^3\\
      1 & \omega^4 & \omega^2 & \omega^6\\
      1 & \omega^6 & \omega^3 & \omega^9 \end{bmatrix} = \begin{bmatrix} F_2 & \begin{bmatrix} 1 \\ & \omega \end{bmatrix} F_2  \\
    F_2  &  \begin{bmatrix} \omega^2 &\\ & \omega^3 \end{bmatrix} F_2\end{bmatrix}
  \end{align}
  Then the relation to Kronecker products becomes clear.
\end{example}

\noindent We list the algebraic properties of butterfly matrices:
\begin{itemize}
\item By induction, it follows that $B\pp{N}$ is orthogonal.
\item For all $N$, $\det B\pp{N} = 1$.
\item If $C_j$ and $S_j$   are diagonal for all $j$, then performing $B\pp{N} x$ requires $O(N \log N)$ operations for $x \in \mathbb R^N$, see \eqref{e:iter-haar}.
\item If $C_j$ and $S_j$ are diagonal for all $j$, then computing $M$ appropriately consecutive rows of $B\pp{N} x $ requires $O(N \log_2 M)$ operations, see \eqref{e:iter}.
\end{itemize}

\begin{definition}
Let $\Sigma = ( (C_j,S_j) )_{j\geq 0}$ be a sequence of pairs of random matrices\footnote{A random matrix is a matrix whose entries are given by random variables.} satisfying
\begin{itemize}
\item $C_{j}$ and $S_{j}$ are symmetric $2^{j}\times 2^{j}$ matrices,
\item $C_j^2 + S_{j}^2 = I$, $C_jS_j = S_jC_j$, and
\item $(C_j,S_j)_{j \geq 0}$ is an independent sequence.  
\end{itemize}
Then a random butterfly matrix $B\pp{N}(\Sigma)$ is given by \eqref{e:buttdef} where $A_1\pp {N/2}$ and $A_2 \pp {N/2}$ are independent and identically distributed (iid) copies of $B\pp{N/2}(\Sigma)$. \end{definition}

\begin{notation}
If $X$ has the same distribution as $Y$ we write $X \sim Y$.
\end{notation}

\begin{definition}
Given a sequence $\Sigma$ as in the previous definition, a simple random butterfly matrix $B_\mathrm{s}\pp{N}(\Sigma)$ is given by \eqref{e:buttdef} where $A_1\pp {N/2} = A_2 \pp {N/2}$ are distributed as $B_\mathrm{s}\pp{N/2}(\Sigma)$.
\end{definition}

In Sections~\ref{s:haarbutt} and \ref{s:general} we make specific choices for the sequence $\Sigma$.

\subsection{Haar measure}

Haar measure is a natural measure on locally compact Hausdorff topological groups.  The proof is originally due to Weil \cite{Weil1951} and can be found translated in \cite{Nachbin1976}.
\begin{theorem}[\cite{Weil1951}]
Let $G$ be a locally compact Hausdorff topological group.  Then, up to a unique multiplicative constant, there exists a unique non-trivial Borel measure $\mu$ such that
\begin{itemize}
\item $\mu(gS)= \mu(S)$ for all $g \in G$ and $S$ a Borel set,
\item $\mu$ is countably additive,
\item $\mu(K) < \infty$ for $K$ compact,
\item $\mu$ is inner regular on open sets and outer regular on Borel sets.
\end{itemize}
\end{theorem}
Let $\mathrm{O}(N)$ denote the group of orthogonal matrices in $\mathbb R^{N \times N}$, let $\mathrm{U}(N)$ denote the group of unitary matrices in $\mathbb C^{N\times N}$ and let $\mathrm{SO}(N) \subset \mathrm{O}(N)$ denote the subgroup of matrices whose determinant is unity.  Since these groups are compact, the associated Haar measure is normalized to be a probability measure.  A matrix sampled from Haar measure on $\mathrm{O}(N)$ is referred to here as a Haar matrix.  It is well-known that one can generate a Haar matrix by applying (modified) Gram-Schmidt to a matrix of iid standard normal random variables \cite{Forrester2010,Mezzadri2006,Stewart1980}.

\subsection{Eigenvalue distributions and tools from random matrix theory}

Given $Q \in \mathrm O (N)$ with eigenvalues $\sigma(Q):= \{\lambda_1,\lambda_2,\ldots, \lambda_N\} \subset \mathbb U : = \{z \in \mathbb C : |z| = 1\}$, define the empirical spectral measure
\begin{align}
\mu_Q = \frac{1}{N} \sum_{j} \delta_{\lambda_j}.
\end{align}
When $Q$ is a random element of $\mathrm O(N)$ we obtain a probability measure on probability measures on $\mathbb U$.  For such a $Q$, define the measure $\mathbb E \mu_Q$ by
\begin{align}
\int_{\mathbb U} f \D \mathbb E \mu_Q : = \mathbb E \left[ \int_{\mathbb U} f \D \mu_Q \right],  \quad f \in C(\mathbb U).
\end{align}
Here $\mathbb E$ refers to the expectation with respect to the distribution of $Q$. The measure $\mathbb E \mu_Q$ is referred to as the density of states.

\begin{definition}
A random orthogonal (or unitary) matrix $Q \in \mathrm{O}(N)$ (or $\mathrm{U}(N)$) has a uniform eigenvalue distribution if $\mathbb E \mu_Q$ is the uniform measure on $\mathbb U$, i.e. the density of states is uniform.
\end{definition}

\begin{theorem}
A random matrix $Q \in \mathrm O(N)$ has a uniform eigenvalue distribution if and only if for $k = 0,1,2,\ldots$
\begin{align}
\mathbb E \left[ \frac{1}{N} \tr Q^k \right] = \begin{cases} 1 & k = 0,\\
0 & k > 0.\end{cases}
\end{align}
\end{theorem}
\begin{proof}
It follows by definition that
\begin{align}
\int_{\mathbb U} z^k \D \mathbb E \mu_Q = \mathbb E \left[\frac{1}{N} \sum_j \lambda_j^k \right] = \mathbb E\left[ \frac{1}{N} \tr Q^k \right].
\end{align}
Uniform measure on $\mathbb U$, given by $\D \mu = \frac{1}{2\pi \I} \frac{\D z}{z}$, is uniquely characterized by the fact that
\begin{align}
\int_{\mathbb U} z^k \D \mu = 0
\end{align}
for $k \in \mathbb Z \setminus \{0\}$ and equal to unity for $k = 0$.  And so, it remains to deduce
\begin{align}
\int_{\mathbb U} z^{-k} \D \mathbb E \mu_Q = 0, \quad k > 0.
\end{align}
But this follows from the fact that $\tr Q^{-k} = \tr (Q^k)^T = \tr Q^k$.
\end{proof}
\begin{definition}
Given a sequence of orthogonal (or unitary) random matrices $(Q\pp{N})_{n \geq 1}$, $Q\pp{N} \in \mathrm O(N)$ (or $\mathrm{U}(N)$), $N = N(n)$, $N(n)$ strictly increasing, then the eigenvalues of the sequence are said to be almost surely uniform if for each $f \in \mathbb C(\mathbb U)$
\begin{align}
\lim_{n \to \infty} \int_{\mathbb U} f \D \mu_{Q\pp N}  = \frac{1}{2 \pi \I}\int_{\mathbb U} f(z) \frac{\D z}{z} \quad \text{almost surely}.
\end{align}
\end{definition}

The following is classical but we prove it for completeness.

\begin{theorem}\label{t:var}
Given a sequence of random matrices $(Q\pp N)_{n \geq 1}$, $Q\pp N \in \mathrm O(N)$, $N = N(n)$, $N(n)$ strictly increasing, assume that $Q \pp N$ has a uniform eigenvalue distribution for each $n$.  Suppose for each $k=1,2,\ldots$ that
\begin{align}
\mathbb E \left[ \frac{1}{N^2}\left( \tr (Q\pp{N})^k \right)^2 \right] \leq C_k n^{-1-c_k},
\end{align}
for some constants $C_k,c_k > 0$.  Then the eigenvalues of the sequence are almost surely uniform.
\end{theorem}
\begin{proof}
First, by the Chebyshev inequality, for $k > 0$
\begin{align}
\mathbb P\left( \frac{1}{N}\left| \tr (Q\pp{N})^k \right| \geq \epsilon \right) \leq \frac{C_k n^{-1-c_k}}{\epsilon^2}.
\end{align}
Then, because
\begin{align}
\sum_{n = 1}^\infty \mathbb P\left( \frac{1}{N}\left| \tr (Q\pp{N})^k \right| \geq \epsilon \right) < \infty,
\end{align}
the Borel--Cantelli lemma implies $\mathbb P \left( \frac{1}{N}\left| \tr (Q\pp{N})^k \right| \geq \epsilon  ~ \text{ infinitely often} \right) = 0$ for every $\epsilon > 0$.  Now let $\epsilon_j \to 0$ be a sequence and $\Omega^c = \bigcup_j \left\{ \frac{1}{N}\left| \tr (Q\pp{N})^k \right| \geq \epsilon_j  ~ \text{ infinitely often} \right\}$.  For $\omega \in \Omega$, for each $j$, $\frac{1}{N}\left| \tr (Q\pp{N}(\omega))^k \right| < \epsilon_j$ for sufficiently large $n$.  Thus $\frac{1}{N}\left| \tr (Q\pp{N})^k \right|$ converges almost surely to zero.  This shows that
\begin{align}
  \frac{1}{N} \tr (Q\pp{N})^k \to  \mathbb E \frac{1}{N} \tr (Q\pp{N})^k~~~\text{ almost surely},
\end{align}
for every $k$, as the expectation on the right is only non-zero for $k = 0$.  

For\footnote{We use $C(U)$ to denote continuous, complex-valued functions on a set $U$.} $f \in C(\mathbb U)$ and a sequence $\epsilon_j \to 0$, approximate
\begin{align}
\sup_{z \in \mathbb U} \left| f(z) - \sum_{\ell = -M_j}^{M_j} a_\ell(j) z^\ell \right| \leq \epsilon_j.
\end{align}
Then if $\mu$ is uniform measure on $\mathbb U$
\begin{align}
\left| \int_{\mathbb U} f \D\mu_{Q\pp{N}} - \int_{\mathbb U} f \D\mu \right| \leq 2 \epsilon_j + \left| \sum_{\ell = -M_j}^{M_j}  \int_{\mathbb U} a_\ell(j) z^\ell (\D\mu_{Q\pp{N}} - \D\mu) \right|.
\end{align}
Since the last term tends to zero almost surely for each $j$ let $\Omega_j$ be the set on which convergence occurs and define $\Omega ' = \bigcup_j \Omega_j$.  Then on  $\Omega'$, $\mathbb P(\Omega') = 1$,
\begin{align}
\limsup_{n \to \infty} \left| \int_{\mathbb U} f \D\mu_{Q\pp{N}} - \int_{\mathbb U} f \D\mu \right| \leq 2 \epsilon_j, \quad \text{for all } j.
\end{align}
The establishes the theorem.
\end{proof}

In particular, if a sequence $(Q\pp N)_{n \geq 1}$, $N = N(n)$, $Q\pp N \in \mathrm O(N)$ is almost surely uniform one can measure the arc length of $A_{\phi,\psi} = \{ \E^{\I \theta}, 0 \leq \phi \leq \theta \leq \psi \leq 2 \pi\}$ by counting either the number of eigenvalues it contains:  If $|U|$ is the cardinality of the set $U$, then
\begin{align}\label{e:as}
\psi - \phi = 2 \pi \lim_{n \to \infty} \frac{| \{ \lambda : \lambda \in \sigma( Q\pp{N} ) \cap A_{\phi,\psi} \} |}{N}~~ \text{ almost surely,}
\end{align}
or by sampling a number of independent copies $(Q\pp{N}_1,Q\pp{N}_2, \ldots)$
\begin{align}\label{e:lln}
\psi - \phi = 2 \pi \lim_{m \to \infty}   \sum_{j=1}^m \frac{| \{ \lambda : \lambda \in \sigma( Q\pp{N}_j ) \cap A_{\phi,\psi} \} |}{mN}~~ \text{ almost surely,}
\end{align}
for fixed $n$ (by the Strong Law of Large Numbers).  It is important to note that \eqref{e:lln} holds because $\mathbb P( \text{more than one eigenvalue in } B) \to 0$ as the measure of $B \to 0$ for any $B \subset \mathbb U$.

The link between the properties of butterfly matrices that we establish here and the properties of the eigenvalues of Haar matrices in $\mathrm{U}(2^n)$ is the most clear. The following theorem is classical.

\begin{theorem}[\cite{Diaconis1994}]
A Haar matrix $U^{(N)} \in \mathrm{U}(N)$ has a uniform eigenvalue distribution.  As $N \to \infty$ the eigenvalues of $U^{(N)}$ are almost surely uniform.  A Haar matrix $Q^{(N)} \in \mathrm{O}(N)$ does not have a uniform eigenvalue distribution for finite $N$, yet it satisfies \eqref{e:as}.  
\end{theorem}

\noindent Almost sure uniformity follows from the fact that $\mathbb E \left[\left| \tr \left( U\pp{N} \right)^k \right|^2 \right] = k$ \cite{Diaconis1994} .       

\section{Haar-butterfly matrices}\label{s:haarbutt}

Consider the class of butterfly matrices denoted by $\mathrm B(2^n)$, $n \geq 1$ and defined recursively by
\begin{align}
\begin{bmatrix} \cos \theta A & \sin \theta A \\ - \sin \theta A & \cos \theta A \end{bmatrix}, \quad 0 \leq \theta \leq 2 \pi,
\end{align}
where $A \in \mathrm B(2^{n-1})$ and  $\mathrm B(1) = \left\{ \begin{bmatrix} 1 \end{bmatrix} \right\}$.  We claim that $\mathrm B(2^n)$ is a subgroup of $\mathrm{SO}(2^n)$.  Indeed, this is clear for $n = 0$.  Assuming the claim for $\mathrm B(2^{n-1})$, let $A,B \in \mathrm B(2^{n-1})$
\begin{align}
\begin{bmatrix} \cos \theta A & \sin \theta A \\ - \sin \theta A & \cos \theta A \end{bmatrix}&\begin{bmatrix} \cos \varphi B & \sin \varphi B \\ - \sin \varphi B & \cos \varphi B \end{bmatrix} \\
&= \begin{bmatrix} AB & 0 \\ 0 & AB \end{bmatrix} \begin{bmatrix}\cos \theta I & \sin \theta I \\ - \sin \theta I & \cos \theta I \end{bmatrix}\begin{bmatrix} \cos \varphi I & \sin \varphi I \\ - \sin \varphi I & \cos \varphi I \end{bmatrix}\\
& = \begin{bmatrix} AB & 0 \\ 0 & AB \end{bmatrix} \begin{bmatrix}\cos (\theta + \varphi) I & \sin (\theta+ \varphi) I \\ - \sin (\theta + \varphi) I & \cos (\theta + \varphi) I \end{bmatrix}\\
& = \begin{bmatrix}\cos (\theta + \varphi) AB & \sin (\theta+ \varphi) AB \\ - \sin (\theta + \varphi) AB & \cos (\theta + \varphi) AB \end{bmatrix} \in \mathrm B(2^n).
\end{align}
Then taking $\varphi = - \theta$ one can see that an inverse element exists. Therefore $\mathrm B(2^n)$ is a subgroup of $\mathrm{SO}(2^n)$ for every $n$.

\begin{definition}
Haar-butterfly matrices are simple random butterfly matrices $B_s\pp{N}(\Sigma_{\mathrm{R}})$ where $\Sigma_{\mathrm{R}} = ( (C_j,S_j) )_{j\geq 0}$, $C_j = \cos \theta_j I, ~ S_j = \sin \theta_j I$ and $(\theta_j)_{j \geq 0}$ are iid uniform on $[0, 2\pi]$.  
\end{definition}

The name in this definition is justified by the following result.
\begin{proposition}\label{p:hb}
The measure induced on $\mathrm B(2^n)$ by Haar-butterfly matrices coincides with Haar measure on $\mathrm B(2^n)$.
\end{proposition}
\begin{proof}
We show left-invariance of the distribution.  Let $B \in \mathrm B(2^n)$ be a Haar-butterfly matrix and let $B'\in \mathrm B(2^n)$ be a butterfly matrix. First, $n = 0$ is clear. Assume the claim for $n-1$.  Then we have 
\begin{align}
B' B = \begin{bmatrix} \cos ( \theta_{n-1} + \theta) AC & \sin ( \theta_{n-1} + \theta) AC \\ - \sin ( \theta_{n-1} + \theta) AC & \cos ( \theta_{n-1} + \theta) AC \end{bmatrix}
\end{align}
where $A \in \mathrm B(2^{n-1})$ and $C$ is a Haar-butterfly matrix in $\mathrm B(2^{n-1})$.  Then $\cos ( \theta_{n-1} + \theta) = \cos ( \theta_{n-1} + \theta \mod 2 \pi)$ and $\theta_{n-1} + \theta \mod 2 \pi$ has the same distribution as $\theta_{n-1}$ and $AC$ is a Haar-butterfly matrix by the inductive hypothesis.  It is not necessary, but if one wants to check inner and outer approximation, it follows directly because this is the smooth forward of uniform measure on $[0,2\pi)^n$, i.e. it is a proper mapping.
\end{proof}

\subsection{Uniformity of eigenvalue distributions}
The following simple lemma immediately applies that Haar-butterfly matrices have a uniform eigenvalue distribution.
\begin{lemma}
Let $A \in \mathrm O(N)$ be any orthogonal matrix.  Let $\theta$ be uniformly distributed on $[0,2\pi]$ and independent of $A$, if  $A$ is random.  Then
\begin{align}
\hat A = \begin{bmatrix} \cos \theta A & \sin \theta A \\ - \sin \theta A & \cos \theta A \end{bmatrix}
\end{align}
has a uniform eigenvalue distribution.
\end{lemma}
\begin{proof}
Because this matrix is orthogonal, it suffices to check that the expectation of the trace of every positive power is zero:
\begin{align}
\hat A^k = \begin{bmatrix} \cos \theta A & \sin \theta A \\ - \sin \theta A & \cos \theta A \end{bmatrix}^k = \begin{bmatrix} \cos k\theta A^k & \sin k\theta A^k \\ - \sin k\theta A^k & \cos k\theta A^k \end{bmatrix}.
\end{align}
From this it directly follows that $\mathbb E_\theta \hat A^k = 0$, and hence the expectation of the trace vanishes.
\end{proof}
The eigenvalues of Haar-butterfly matrices are also almost surely uniform.
\begin{theorem}  
The sequence $(Q\pp{N})_{n \geq 1}$, $N = 2^n$, where $Q\pp{N} \sim B_s\pp{N}(\Sigma_{\mathrm{R}})$ has eigenvalues that are almost surely uniform.
\end{theorem}
\begin{proof}
By Theorem~\ref{t:var} and Proposition~\ref{p:hb} it suffices to estimate the expectation of the trace of powers squared.  We have
\begin{align}\label{e:tr-struct}
\tr (Q\pp{N})^k = 2 \cos k \theta_{n-1} \tr (X)^k, \quad X \sim B_s\pp{n-1}(\Sigma_{\mathrm{R}}),
\end{align}
 where $\tr (Q\pp{N})^k$ and $X$ are independent. Therefore by first taking an expectation with respect to $\theta_{n-1}$
\begin{align}
\mathbb E  \left[ \left( \tr (Q\pp{N})^k \right)^2 \right] = 2 \mathbb E \left[ \left( \tr (Q\pp{N/2})^k \right)^2 \right].
 \end{align}
 From this we have
 \begin{align}\label{e:intout}
 \frac{1}{2^{2n}}\mathbb E\left[ \left( \tr (Q\pp{N})^k \right)^2 \right] = \frac{1}{2}  \frac{1}{2^{2n-2}} \mathbb E \left[ \left( \tr (Q\pp{n-1})^k \right)^2 \right],
 \end{align}
 or since $N = 2^{n}$, $\frac{1}{N^2}\mathbb E \left[ \left( \tr (Q\pp{N})^k \right)^2 \right] = 2^{-n}$ and the theorem follows from Theorem~\ref{t:var}.
\end{proof}

\subsection{Joint distribution on the eigenvalues}

Another natural question to ask is that of the joint distribution of the eigenvalues.  And while the uniformity results might lead one to speculate that the eigenvalues are iid on $\mathbb U$ or share a distribution related to the eigenvalues of a Haar matrix, one can see that dimensionality immediately rules this out: $\mathrm B(2^{n})$ is a manifold of dimension $n$.  What is true though, is that $n$ of the eigenvalues of $Q\pp{N} \sim B_s\pp{N}(\Sigma_{\mathrm{R}})$ are iid uniform on $\mathbb U$.

\begin{lemma}
Let $A \in \mathrm O(N)$ and $0 \leq \theta < 2 \pi$.  Then the arguments of the eigenvalues of
\begin{align}
\hat A = \begin{bmatrix} \cos \theta A & \sin \theta A \\ - \sin \theta A & \cos \theta A \end{bmatrix}
\end{align}
are given by $(\theta_j \pm \theta)_{j \geq 1}$ where $(\theta_j)_{j \geq 1}$ are the arguments of the eigenvalues of $A$.
\end{lemma}
\begin{proof}
Consider the characteristic polynomial
\begin{align}
\det (\hat A - \lambda I) = \det \begin{bmatrix} \cos \theta A - \lambda I & \sin \theta A \\ - \sin \theta A & \cos \theta A - \lambda I\end{bmatrix} = \det ( A^2    - 2 \lambda \cos \theta A + \lambda^2 I).
\end{align}
Now, let $v$ be an eigenvector for $A$ with eigenvalue $\lambda_1$.  Then
\begin{align}
( A^2 - 2 \lambda \cos \theta A + \lambda^2 I)v = (\lambda_1^2  - 2 \lambda \lambda_1 \cos \theta  + \lambda^2) v = 0
\end{align}
if $\lambda = \E^{ \pm \I \theta } \lambda_1$. This establishes the lemma.
\end{proof}

This shows that the eigenvalues of $Q\pp{N} \sim B_s\pp{N}(\Sigma_{\mathrm{R}})$ are given by
\begin{align}
\exp \left( \I ( \pm \theta_0 \pm \theta_2 \pm \cdots \pm \theta_{n-1}) \right),
\end{align}
for all possible $2^n$ choices of signs where $(\theta_j)_{j \geq 0}$ are iid uniform on $[0, 2\pi]$ and coincide with the angles in the definition of $\Sigma_{\mathrm{R}}$. Such a formula arises in the analysis of the independence of Rademacher functions \cite[Chapter 1]{Kac1959}.

We pick a canonical set of angles $\hat \theta_0 = \sum_j \theta_j$
\begin{align}
\hat \theta_j = \hat \theta_0 - 2 \theta_j, \quad j \geq 1.
\end{align}
Then define $x_j = \cos \hat \theta_j$.  It follows that the joint density of these variables $\rho(x_0,\ldots,x_{n-1})$ is given by
\begin{align}
\rho(x_0,\ldots,x_{n-1}) = \frac{1}{\pi^n} \prod_{j=0}^{n-1} \frac{1}{\sqrt{1-x_j^2}} \mathbbm 1_{(-1,1)}(x_j).
\end{align}
It is convenient to use the $x_j$ variables because each $x_j$ corresponds to two eigenvalues $\pm \theta_j = \pm \cos^{-1} x_j$.  The remaining eigenvalues are determined directly from the $x_j$'s.  This demonstrates a clear departure from Haar measure on $\mathrm O(N)$ where the probability of two eigenvalues being close is much smaller \cite{Forrester2010}.

\subsection{Failure of a Central Limit Theorem for linear statistics}

For Haar matrices $U^{(N)} \in \mathrm{U}(N)$, from \cite{Diaconis1994,Diaconis2001}, it follows that $\mathbb E_{U^{(N)}}\left[\tr \left( U^{(N)} \right)^k \right] = 0$ and $\mathbb E \left[ \left| \tr \left( U^{(N)} \right)^k \right|^2 \right] = k$.  And thus it is natural to ask about convergence of
\begin{align}
  \frac{\tr (U^{(N)})^k - \mathbb E\left[\tr \left( U^{(N)} \right)^k \right]}{\sqrt{\mathrm{Var}\left( \tr (U^{(N)})^k \right)}} = \frac{\tr (U^{(N)})^k}{\sqrt{k}}  \quad \text{as} \quad N = 2^n \to \infty.
\end{align}
Indeed, it follows that this converges to a standard (complex) normal random variable \cite{Diaconis2001}.  Once can then ask about convergence of \emph{linear statistics:}
\begin{align}
  \sum_k a_k \tr (U^{(N)})^k.
\end{align}
These statistics can also be shown to converge, after appropriate scaling, to normal distributions \cite{Diaconis2001}.  This is called the central limit theorem (CLT) for linear statistics.  We now demonstrate that this fails to hold for Haar-butterfly matrices.

Let $Q\pp{N}$ be a Haar-butterfly matrix. One can immediately notice that Haar-butterfly matrices are different via the relation
\begin{align}
\mathbb E \left[ \left( \tr \left( Q^{(N)} \right)^k \right)^2 \right] = N, \quad N = 2^n.
\end{align}
Less cancellation forces this to grow, compared to $U\pp{N}$.  We now examine the convergence of
\begin{align}\label{e:clt}
 \frac{\left(\tr \left(Q^{(N)} \right)^k \right)^2}{N} \quad \text{as} \quad N = 2^n \to \infty,
\end{align}
which would have to converge to a chi-squared distribution if a CLT holds.  Simple considerations demonstrate that this is a martingale with respect to the sequence $(\theta_j)_{j=0}^\infty$ by considering \eqref{e:tr-struct} and basically following \eqref{e:intout}
\begin{align}
\frac{1}{N} \mathbb E_{\theta_{n-1}} \left( \tr\left(Q^{(N)} \right)^k \right)^2  =  \frac{2}{N} \left( \tr \left(Q^{(N/2)} \right)^k \right)^2
\end{align}
This also easily follows from the relation $\left(Q^{(N)} \right)^k = N \prod_{j=0}^{n-1} \cos k \theta_j$.   Motivated by this formula, we apply the Strong Law of Large Numbers to
\begin{align}
  \log \frac{\left( \tr \left(Q^{(N)} \right)^k\right)^2}{N} = \sum_{j=0}^{n-1} \log \cos^2 k\theta_j + n \log 2.
\end{align}
Then, using $\int_{0}^{2\pi} \log \cos^2 k\theta \frac{\D \theta}{2 \pi} = - 2 \log 2$, we have that as $n \to \infty$
\begin{align}
  \frac{1}{n} \sum_{j=0}^{n-1} \log \cos^2 k\theta_j \to - 2 \log 2 \quad \text{a.s.},\\
  \frac{1}{n} \log \frac{\left(\tr \left(Q^{(N)} \right)^k  \right)^2}{N} \to - \log 2 \quad \text{a.s.}.
\end{align}
Therefore \eqref{e:clt} tends to zero almost surely, and a CLT for linear statistics does not hold.

We note that if the eigenvalues of an orthogonal matrix $O\pp{N}$ satisfy $\lambda_j = \E^{i \theta_j}$ for $j = 1,2,\ldots,N$, where $(\theta_j)$ are iid on $[0,2\pi)$, then \eqref{e:clt} will converge to a non-trivial distribution.  So we see that \eqref{e:clt} is small in the case of Haar matrices (its $o(1)$), large for iid eigenvalues (it converges in distribution) and it sits somewhere between for Haar-butterfly matrices.  So, this matrix is in this sense ``less random'' than $U\pp{N}$ more random than $O\pp{N}$.  This is an indication it might be more effective at randomizing a linear system into a generic state than $O\pp{N}$.


\section{Another class of random butterfly matrices}\label{s:general}

Now consider the random butterfly matrices (no longer simple) defined by $B\pp{N}(\Sigma_{\mathrm{R}})$:
\begin{align}\label{e:rbf}
Q\pp{N} = \begin{bmatrix} \cos \theta_{n-1} A & \sin \theta_{n-1} B \\
-\sin \theta_{n-1} A & \cos \theta_{n-1} B \end{bmatrix},
\end{align}
where $A,B \sim B \pp{N/2}(\Sigma_{\mathrm{R}})$ are independent.  We refer to these as \emph{non-simple} butterfly matrices.  There is also no longer a natural group structure, but ``more randomness'' has been injected into the matrix.

\subsection{Uniformity of the eigenvalue distributions.}

To analyze the distribution of the entries of a $B\pp{N}(\Sigma_{\mathrm{R}})$ matrix, we think of it as a $2^{n-1} \times 2^{n-1}$ matrix of $2\times 2$ blocks.  The block at location $(i,j)$ is of the form
\begin{align}
\pm \left(\prod_{i=1}^m \cos \theta_j^{(i,j)} \right) \left( \prod_{j=m}^{n} \sin \theta_j^{(i,j)}  \right) R(\theta_0^{(j)}), \quad R(\theta) := \begin{bmatrix}
\cos \theta & \sin \theta \\ - \sin \theta & \cos \theta
\end{bmatrix},
\end{align}
for some $1\leq m \leq n$.  In other words, each $2\times 2$ block of the matrix is a product of sines and cosines that depend on both $(i,j)$, multiplied by a rotation matrix that depends only on $j$.  We use the notation
\begin{align}\label{e:cij}
c^{(n)}_{i,j} = c_{i,j} = \pm \left(\prod_{j=1}^m \cos \theta_j^{(i,j)} \right) \left( \prod_{j=m}^{n} \sin \theta_j^{(i,j)}  \right).
\end{align}
From this representation, we immediately obtain the following.
\begin{lemma}
Assume $Q\pp{N} \sim B\pp{N}(\Sigma_{\mathrm{R}})$ then $\mathbb E ( Q\pp{N})^k = 0$ for every $k$, and hence the eigenvalues of $Q\pp{N}$ are uniform.
\end{lemma}
\begin{proof}
The $(i,j)$ $2 \times 2$ block of $( Q\pp{N})^k$ is given by
\begin{align}
( Q\pp{N})^k_{ij} &= \sum_{1 \leq j_1, j_2, \ldots, j_{k-1}\leq 2^{n-1}} c_{j,j_1} c_{j_1,j_2}\cdots c_{i,j_{k-1}} R\left(\theta_0^{(j)} + \sum_{\ell=1}^{k-1}\theta_0^{(j_\ell)}  \right),
\end{align}
where the sum is over all points $(j_1,j_2,\ldots,j_{k-1}) \in \{1,2,\ldots,2^{n-1}\}^{k-1}$.  By the construction of $B\pp{N}(\Sigma_{\mathrm{R}})$ the $\theta_0^{(j)}$ variables are independent of 
$\theta_j^{(k,\ell)}$ for every choice of $j,k,\ell$.  And so we take the expectation over the $\theta_0^{(j)}$ variables first to find
\begin{align}
\mathbb E_{\theta_0^{(j)}} ( Q\pp{N})^k_{ij} = 0,
\end{align}
since
\begin{align}
\mathbb E_{\theta_0^{(j)}} R\left(\theta_0^{(j)} + \sum_{\ell=1}^{k-1}\theta_0^{(j_\ell)}  \right) = 0.
\end{align}
\end{proof}
Because these matrices do not give a significant performance difference over the Haar-butterfly matrices in Section~\ref{s:numerics}, we do not pursue their properties further.


\section{Numerical experiments}\label{s:numerics}

We now test classes of matrices by measuring their ability to truly randomize matrix $A \in \mathbb R^{N \times M}$, $N \gg M$.  For such a matrix $A$, assuming it is of full rank, define its $QR$ factorization $A = QR$, $Q \in \mathbb R^{N \times M}$, $R \in \mathbb R^{M\times M}$ where $R$ is upper triangular with positive diagonal entries and $Q$ has orthonormal columns.  The coherence of $A$ is defined by
\begin{align}
  \mathrm{coh}(A) : = \max_{1 \leq j \leq N} \| e_j^T Q\|_2^2, \quad A = QR,
\end{align}
where $(e_j)_{j=1}^M$ is the standard basis for $\mathbb R^M$.  The range of coherence is easily seen to be $[M/N,1]$ (see \cite{Ipsen2014}, for example).  One use of coherence is the following.  If $A$ has small coherence, $N \gg M$, then when solving $Ax = b$ randomly selecting a subset of the rows gives a good approximation of the least-squares solution.  If the coherence is high, such an approximation can fail to hold. As outlined in the introduction, to randomize a linear system take $Ax = b$, randomize it with a orthogonal (or unitary) matrix $\Omega$: $\Omega A x = \Omega b$.  Then the hope is that $\Omega A$ has smaller coherence than $A$ and by subsampling rows one can easily approximate the least-squares solution.  Indeed, in the \emph{Blendenpik} algorithm \cite{Avron2010}, this procedure is used to generate a preconditioner for the linear system, not just an approximate solution.  We call this \emph{randomized coherence reduction}.  We examine the ability of the following matrices to achieve this:
\begin{enumerate}
\item \underline{Haar-butterfly DCT (HBDCT)}: $\Omega = Q_{\mathrm{DCT}} Q\pp{N}$ where $Q_{\mathrm{DCT}}$ is the Discrete Cosine Transform (DCT) matrix and $Q\pp{N}$ is a Haar-butterfly matrix
\item \underline{RBM DCT (RBDCT)}: $\Omega = Q_{\mathrm{DCT}} Q\pp{N}$ where $Q_{\mathrm{DCT}}$ is the Discrete Cosine Transform (DCT) matrix and $Q\pp{N} \sim B\pp{N}(\Sigma_{\mathrm R})$ as given in \eqref{e:rbf}.
\item \underline{Random DCT (RDCT)}:  $\Omega = Q_{\mathrm{DCT}} D$ where $D$ is a diagonal matrix with $\pm 1$ on the diagonal (independent, $\mathbb P( D_{jj} = \pm 1) = 1/2$).  This was considered in \cite{Avron2010}.
\item \underline{Haar matrix}: $\Omega$ is a Haar matrix on $\mathrm O(N)$.
\end{enumerate}
We try to reduce the coherence of the following matrices $A = (a_{ij})_{1 \leq i \leq N, 1 \leq j \leq M}$.
\begin{enumerate}
\item {\tt randn}: $a_{11}$ and $a_{ij}$, $1 \leq i \leq N$, $2 \leq j \leq M$ are iid standard normal random variables and $a_{j1} = 0$ for $j > 1$.  This matrix was considered in \cite{Avron2010} and $\mathrm{coh}(A) = 1$ a.s..
\item {\tt hilbert}: $a_{ij} = 1/(i + j -1)$.  Simple numerical experiments  demonstrate that $\mathrm{coh}(A) \to 1$ as $N, M \to \infty$, rapidly.  This is a classical example of an ill-conditioned matrix.
\end{enumerate}


\subsection{{\tt hilbert} matrices}

In Table~\ref{t:hilbert-mean} we present the sample means of the coherence of the matrix $\Omega A$ where $A$ is the {\tt hilbert} matrix and $\Omega$ is one the four choices of random orthogonal matrices detailed above in this section.  In Table~\ref{t:hilbert-std} we give the standard deviations.  Because of the computational cost, we only give sample means and standard deviations for Haar matrix for $n \leq 12$. It is notable that the HBDCT outperforms the RDCT.  In Figures~\ref{f:g2} and \ref{f:g3} we plot the histograms for the coherence.

\begin{table}[h]
\begin{tabular}{|c|c|c|c|c|c|c|c|c|c|c|} 
  \hline
 $n$ &  $9$ & $10$ & $11$ & $12$ & $13$ & $14$ & $15$ & $16$ & $17$\\
  \hline
  HBDCT& $0.900$ & $0.826$ & $0.752$ & $0.695$ & $0.686$ & $0.725$ & $0.793$ & $0.860$ & $0.912$\\
  RBDCT& $0.887$ & $0.810$ & $0.733$ & $0.675$ & $0.668$ & $0.713$ & $0.787$ & $0.857$ & $0.911$\\
RDCT& $0.916$ & $0.857$ & $0.786$ & $0.737$ & $0.746$ & $0.812$ & $0.877$ & $0.927$ & $0.957$\\
Haar &  $0.281$ & $0.147$ & $0.076$ & $0.039$ & -- & -- & -- & -- & --\\
  \hline
\end{tabular}
\caption{ \label{t:hilbert-mean} Sample means for the coherence of $\Omega A$ where $A \in \mathbb R^{2^n \times M}$ is the {\tt hilbert} matrix when $M = 100$.  Random butterfly matrices give a small improvement when compared to the Random DCT. But the Haar matrix gives a clear improvement over the others.  There is minimal difference between matrices from $B_s\pp{N}(\Sigma_{\mathrm R})$ and $B\pp{N}(\Sigma_{\mathrm R})  $.}
\end{table}

\begin{table}[h]
\begin{tabular}{|c|c|c|c|c|c|c|c|c|c|c|} 
  \hline
 $n$ &  $9$ & $10$ & $11$ & $12$ & $13$ & $14$ & $15$ & $16$ & $17$\\
  \hline
HBDCT& $0.037$ & $0.057$ & $0.071$ & $0.079$ & $0.081$ & $0.084$ & $0.08$ & $0.067$ & $0.049$\\
RBDCT&$0.042$ & $0.061$ & $0.074$ & $0.079$ & $0.081$ & $0.083$ & $0.079$ & $0.065$ & $0.048$\\
RDCT&$0.034$ & $0.05$ & $0.062$ & $0.069$ & $0.071$ & $0.061$ & $0.045$ & $0.03$ & $0.019$\\
Haar&$0.033$ & $0.021$ & $0.015$ & $0.012$ & -- & -- & -- & -- & --\\
  \hline
\end{tabular}
\caption{\label{t:hilbert-std} Sample standard deviations for the coherence of $\Omega A$ where $A \in \mathbb R^{2^n \times M}$ is the {\tt hilbert} matrix  when $M = 100$.  Random butterfly matrices have slightly larger standard deviations when compared to the Random DCT. But the Haar matrix gives a clear improvement over the others.  There is minimal difference between matrices from $B_s\pp{N}(\Sigma_{\mathrm R})$ and $B\pp{N}(\Sigma_{\mathrm R})  $.}
\end{table}

\begin{figure}[h]
  \centering
  \begin{overpic}[width=.7\linewidth]{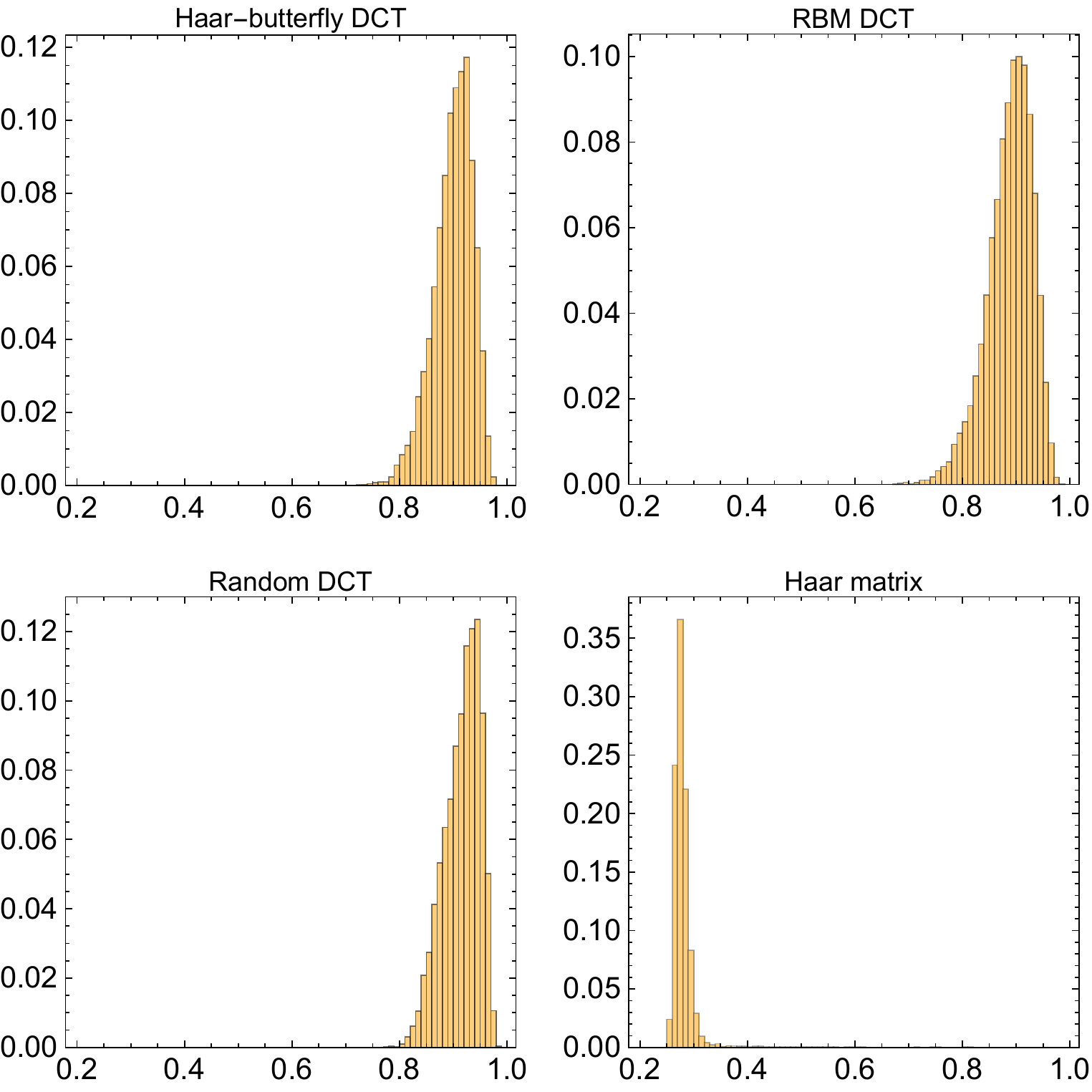}
    \put(42,-5){Coherence}
    \put(-5,37){\rotatebox{90}{Relative frequency}}
  \end{overpic}
  \vspace{.1in}
  \caption{\label{f:g2} Histograms for the coherence of $\Omega A$ when $A \in \mathbb R^{2^n\times M}$ when $A$ is the {\tt hilbert} matrix and $n = 9$, $M = 100$ with 10,000 samples. The Haar matrices out-perform the other random orthogonal matrices but are computationally prohibitive to use in practice.}
\end{figure}

\begin{figure}[h]
  \centering
  \begin{overpic}[width=.95\linewidth]{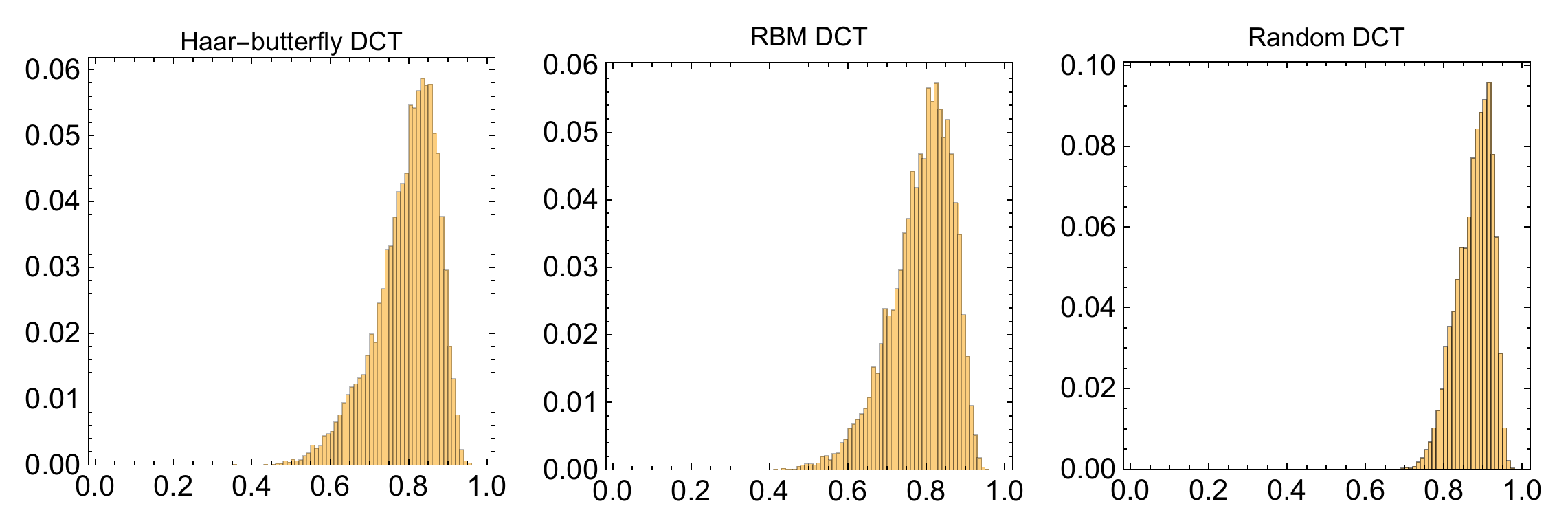}
    \put(42,-3){Coherence}
    \put(-2,5){\rotatebox{90}{Relative frequency}}
  \end{overpic}
  \vspace{.1in}
  \caption{\label{f:g3} Histograms for the coherence of $\Omega A$ when $A \in \mathbb R^{2^n\times M}$ when $A$ is the {\tt hilbert} matrix and $n = 15$, $M = 100$ with 10,000 samples. The random butterfly matrices out-perform the random DCT.  }
\end{figure}

\subsection{{\tt randn} matrices}

In Table~\ref{t:randn-mean} we present the sample means of the coherence of the matrix $\Omega A$ where $A$ is the {\tt randn} matrix and $\Omega$ is one the four choices of random orthogonal matrices.  Similarly, in Table~\ref{t:randn-std} we give the standard deviations.  As before, because of the computational cost, we only give sample means and standard deviations for Haar matrix for $n \leq 12$. This is similar to the experiment performed in \cite{Avron2010}, where the random DCT performs very well. In Figure~\ref{f:g1} we plot the histograms for the coherence.  We see that butterfly matrices do not perform better than the random DCT.  Future work will be in the direction of understanding this at a deeper level.  

\begin{table}[h]
\begin{tabular}{|c|c|c|c|c|c|c|c|c|c|c|} 
  \hline
 $n$ &  $9$ & $10$ & $11$ & $12$ & $13$ & $14$ & $15$ & $16$ & $17$\\
  \hline
HBDCT &$0.285$ & $0.156$ & $0.088$ & $0.051$ & $0.031$ & $0.019$ & $0.012$ & $0.008$ & $0.005$\\
RBDCT &$0.285$ & $0.157$ & $0.088$ & $0.051$ & $0.031$ & $0.019$ & $0.012$ & $0.008$ & $0.005$\\
RDCT&$0.277$ & $0.145$ & $0.075$ & $0.039$ & $0.02$ & $0.01$ & $0.005$ & $0.003$ & $0.001$\\
  Haar&$0.277$ & $0.145$ & $0.075$ & $0.039$ & -- & -- & -- & -- & -- \\
  \hline
\end{tabular}
\caption{\label{t:randn-mean} Sample means for the coherence of $\Omega A$ where $A \in \mathbb R^{2^n \times M}$ is the {\tt randn} matrix  when $M = 100$.  Random butterfly matrices do not give an improvement when compared to the Random DCT. Surprisingly, Haar matrices and the random DCT give similar performance.}
\end{table}

\begin{table}[h]
\begin{tabular}{|c|c|c|c|c|c|c|c|c|c|c|} 
  \hline
 $n$ &  $9$ & $10$ & $11$ & $12$ & $13$ & $14$ & $15$ & $16$ & $17$\\
  \hline
HBDCT &$0.022$ & $0.022$ & $0.02$ & $0.016$ & $0.012$ & $0.009$ & $0.007$ & $0.005$ & $0.004$\\
RBDCT &$0.023$ & $0.022$ & $0.02$ & $0.017$ & $0.012$ & $0.01$ & $0.007$ & $0.005$ & $0.004$\\
RDCT &$0.011$ & $0.006$ & $0.003$ & $0.001$ & $0.001$ & $0.0$ & $0.0$ & $0.0$ & $0.0$\\
Haar &$0.011$ & $0.006$ & $0.003$ & $0.001$ & -- & -- & -- & -- & --\\
  \hline
\end{tabular}
\caption{\label{t:randn-std} Sample standard deviations for the coherence of $\Omega A$ where $A$ is the {\tt randn} matrix  when $M = 100$.  Again, random butterfly matrices do not give an improvement when compared to the Random DCT. Surprisingly, Haar matrices and the random DCT give similar performance --- a very small standard deviations.}
\end{table}

\begin{figure}[h!]
  \centering
  \begin{overpic}[width=.7\linewidth]{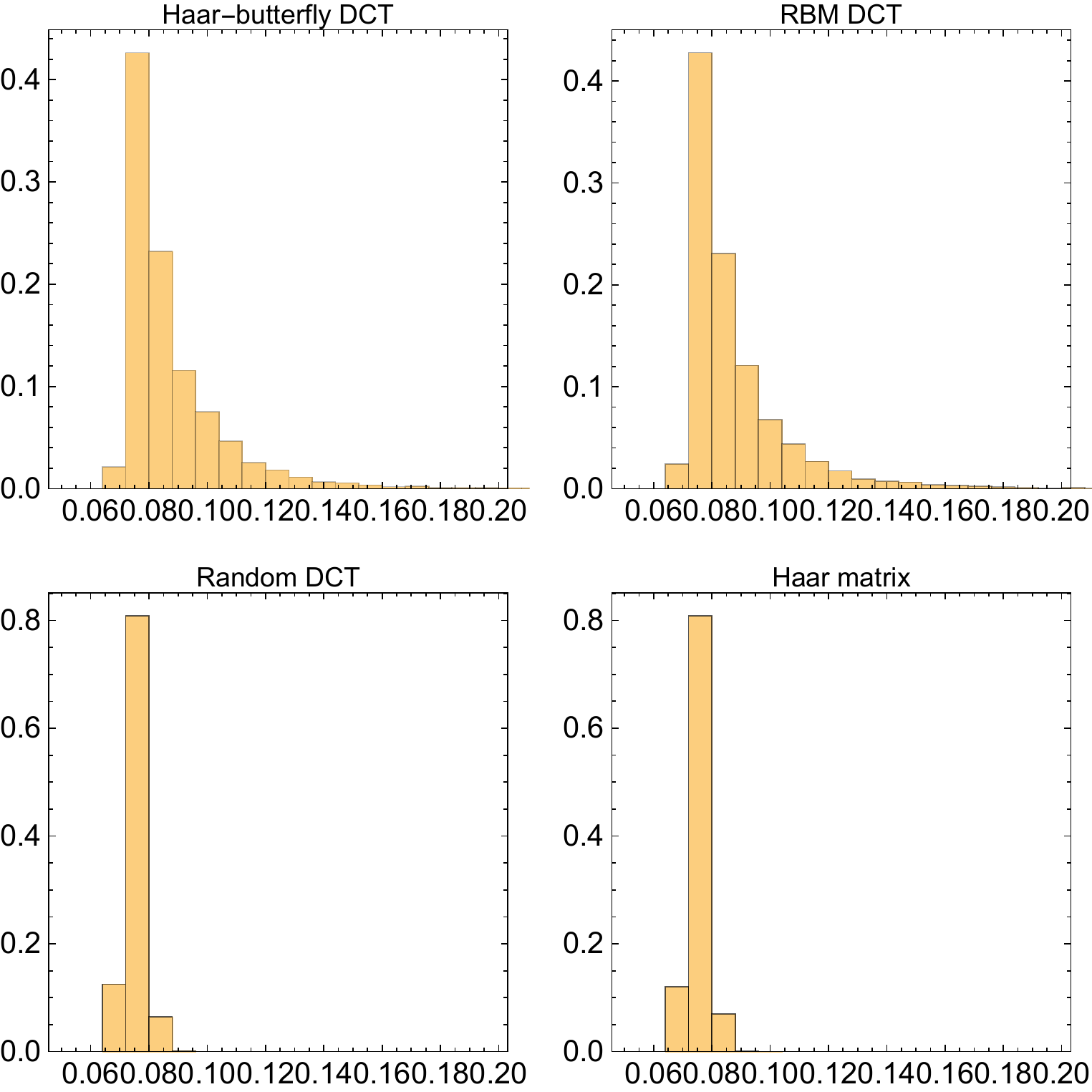}
    \put(42,-5){Coherence}
    \put(-5,37){\rotatebox{90}{Relative frequency}}
  \end{overpic}
  \vspace{.1in}
  \caption{\label{f:g1} Histograms for the coherence of $\Omega A$ when $A \in \mathbb R^{2^n\times M}$ when $A$ is the {\tt randn} matrix and $n = 11$, $M = 100$ with 10,000 samples. The Haar matrices and random DCT matrices perform similarly.}
\end{figure}


\section{Computational methods}

In this section we give {\tt Julia} code to efficiently multiply the discussed random matrices.  We also discuss the complexity of these algorithms.

\subsection{Computing multiplication by a Haar matrix on $\mathrm{O}(N)$}\label{s:haar}

It is well-known that a Haar matrix on $\mathrm{O}(N)$ can be sampled by computing the QR factorization of an $N \times N$ matrix of iid standard normal random variables requiring $O(N^3)$ operations.  But if one wants to multiply a sampled Haar matrix and a vector, the matrix does not need to be constructed and the complexity can be reduced to $O(N^2)$ operations \cite{Stewart1980}.  We discuss the approach briefly and give {\tt Julia} code.  Given a non-trivial vector $u \in \mathbb R^L$, $1 < L \leq N$ define the Householder reflection matrix\footnote{Classically, one usually uses $u \pm \|u\|_2 e_1$ to maximize numerical stability but in our random setting this will not (with high probability) be a numerical issue.}
\begin{align}
  H(u) := \begin{bmatrix} I_{(N-L)\times (N-L)} & 0 \\ 0 & I_{L\times L} - 2 vv^T \end{bmatrix}, \quad v = \frac{u - \|u\|_2 e_1}{\|u - \|u\|_2 e_1\|_2}.
\end{align}
If $M =1$ then $H(u) : = \mathrm{diag}(1,1,\ldots,1,\mathrm{sign}(u))$.  Let $u_1,u_2,\ldots,u_{N}$ be iid random vectors with standard normal entries.  Further, suppose that $u_j \in \mathbb R^{N-j+1}$.  Then the matrix
\begin{align}
  H(u_N) H(u_{N-1}) \cdots H(u_1)
\end{align}
is a Haar matrix. From a computational point of view, $H(u_j)$ requires $O(N-j)$ operations to apply to a vector, so that the product requires only $O(N^2)$ operations to apply to a vector, down from the $O(N^3)$, that is required for the QR factorization approach.  To sample an entire Haar matrix one can simply call {\tt haarO(eye(N))}.

\begin{lstlisting}[linewidth = \linewidth,frame = trBL]%
  
function house!(v::Vector,A::Array)
  n = length(v)
  u = v/(norm(v)/sqrt(2))
  A[end-n+1:end,:] = A[end-n+1:end,:]-
  u*(u'*A[end-n+1:end,:])
end

function haarO(A::Array)
  m = size(A)[1]
  B = A
  for i = 1:m-1
    v = randn(m+1-i)
    v[1] = v[1] - norm(v)
    house!(v,B)
  end
  B[end,:] = sign(randn())*B[end,:]  
  B
end
\end{lstlisting}

\subsection{Computing multiplication by a Haar-butterfly  matrix}\label{s:code_haarrbm}

The function {\tt haar\_rbm(v)} below applies a Haar-butterfly matrix to the vector {\tt v}.  This is done recursively in the following implementation.

\begin{lstlisting}[linewidth = \linewidth,frame = trBL]%[float,caption=A floating example]
  
f = n -> 2*pi*rand(n)

function haar_rbm(t::Vector,v::Array)
  if length(t) == 0
    return v
  end
  n = size(v)[1]
  m = Int(n/2)
  v1 =  haar_rbm(t[1:end-1],v[1:m,:])
  v2 =  haar_rbm(t[1:end-1],v[m+1:end,:])
  c = cos(t[end])
  s = sin(t[end])
  return vcat(c*v1+s*v2,-s*v1+c*v2)
end

function haar_rbm(v::Array)
  n = Int(log(2,size(v)[1]))
  return haar_rbm(f(n),v)
end

\end{lstlisting}

If one wants to invert the transformation, the angles $\theta_j$ can be saved as the following code illustrates.

\begin{lstlisting}[linewidth = \linewidth,frame = trBL]
julia> n = 5; v = ones(Float64,2^n); t = f(n); norm(v)
5.656854249492381

julia> hatv = haar_rbm(t,v); norm(hatv)
5.6568542494923815 

julia> norm(v-haar_rbm(-t,hatv))
2.462593796972316e-15 
\end{lstlisting}

If the entire matrix is to be constructed, one can just call {\tt haar\_rbm(eye(2\string^n))}. To compute the number of operations (we concentrate on multiplications). Let $O_n$ denote the number of multiplications it requires to apply this algorithm to a $2^n$-dimensional vector.  Then $O_n$ satisfies
\begin{align}
  O_n = 2 O_{n-1} + 2^{n}, \quad O_0 = 0.
\end{align}
It is straightforward to verify that
\begin{align}\label{e:iter-haar}
  O_n = n 2^{n+1} = 2N \log_2 N.
\end{align}

\newcommand{\mtt}[1]{\text{\tt#1}}

The following code implements subsampled Haar-butterfly matrix multiplication.  Given $k$, divide a vector $w \in \mathbb R^{2^n}$ into $2^j$ vectors $w_\ell$, $\ell = 1,2,\ldots,2^k$ of size $M = 2^{n-k}$.  For a Haar-butterfly matrix $Q\pp{N}$, $N = 2^n$, $w = Q\pp{N}\mtt v$, the function {\tt harr\_rbm(t,v,k,j)} returns $w_\ell$ such that $\mtt j \in \left[(\ell-1) 2^{n-k} + 1, \ell 2^{n-k}\right]$.  In other words, it returns the vector $w_\ell$ that contains the entry $\mtt j$ of $w$.

\begin{lstlisting}[linewidth = \linewidth,frame = trBL]
function haar_rbm(t::Vector,v::Array,j::Int,k::Int)
  if k == 0
    return haar_rbm(t,v)
  end
  if length(t) == 0
    return v
  end
  n = size(v)[1]
  m = Int(n/2)
  c = cos(t[end])
  s = sin(t[end])
  if j > m
    v1 =  haar_rbm(t[1:end-1],v[1:m,:],j-m,k-1)
    v2 =  haar_rbm(t[1:end-1],v[m+1:end,:],j-m,k-1)
    return -s*v1+c*v2
  else
    v1 =  haar_rbm(t[1:end-1],v[1:m,:],j,k-1)
    v2 =  haar_rbm(t[1:end-1],v[m+1:end,:],j,k-1)
    return c*v1+s*v2
  end
end
\end{lstlisting}
The code is demonstrated here.  A small speedup is realized.

\begin{lstlisting}[linewidth = \linewidth,frame = trBL]
julia> n = 15; N = 2^n; t = f(n); v = ones(Float64,N);
 j = 1; k = 13;

julia> @btime c1 = haar_rbm(t,v,j,k); #Subsampled
  22.499 ms (581635 allocations: 39.27 MiB)

julia> m = 2^(n-k);
@btime c2 = haar_rbm(t,v)[1+(j-1)*m:j*m]; #Original
  29.692 ms (655603 allocations: 55.44 MiB)

julia> norm(c1-c2)
0.0
\end{lstlisting}

Given $n$ and $0 \leq k \leq n$ we compute $O_{n,k}$ which is the number of multiplications required to apply a subsampled Haar-butterfly matrix to a vector with $2^n$ entries.  If $k = 0$ we see that $O_{n,0} = O_n$ (the un-subsampled version).  With $n$ fixed, we find the following recursion
\begin{align}
  O_{n,k} = 2 O_{n-1,k+1} + 2^{n-k+1}, \quad n = k, k+1, \ldots.
\end{align}
From this it follows
\begin{align}\label{e:iter}
  O_{n,k} = 2^n \left( 2(n-k) + \sum_{\ell = 0}^{k-1} 2^{-\ell} \right) \leq 2^{n+1}(n - k + 1) = 2 N (\log_2 M +1).
\end{align}

\subsection{Computing multiplication by a non-simple random butterfly matrix}\label{s:code_rbm}

For the non-simple random butterfly matrices, in this implementation, we do not keep track of the sampled random variables which would allow inversion.

\begin{lstlisting}[linewidth = \linewidth,frame = trBL]
function rbm(v::Array)
  n = size(v)[1]
  if n == 1
    return v
  end
  m = Int(n/2)
  t = 2*pi*rand()
  v1 =  rbm(v[1:m,:])
  v2 =  rbm(v[m+1:end,:])
  c = cos(t)
  s = sin(t)
  return vcat(c*v1+s*v2,-s*v1+c*v2)
end
\end{lstlisting}
The code is demonstrated here.

\begin{lstlisting}[linewidth = \linewidth,frame = trBL]
julia> n = 5; v = ones(Float64,2^n); norm(v)
5.656854249492381

julia> norm(rbm(v))
5.65685424949238

julia> norm(rbm(v)-v)
7.184366533253482
\end{lstlisting}

\section{Conclusion}

When one is tasked with randomizing a linear system to force it into its ``most generic'' state, the theoretical choice is to multiply the matrix by a matrix distributed according to Haar measure on the unitary (or orthogonal) group.  But this is computationally prohibitive.  The butterfly matrices give an alternative, requiring only $O(N \log N)$ operations for a matrix-vector product.  The theory we have developed suggests that the Haar-butterfly matrices and the non-simple butterfly matrices mimic Haar matrices better in some respects than matrices with iid eigenvalues.  We summarize the theoretical distinction in the Table~\ref{table}.

The numerical results in Section~\ref{s:numerics} demonstrate that random butterfly matrices mimic Haar matrices in other respects.  The Haar-butterfly matrices work almost as well as the randomized DCT (RDCT) when trying to minimize the coherence for the {\tt randn} matrix and out-perform the RDCT for the {\tt hilbert} matrix.  Both the RDCT and the Haar-butterfly matrices are out-performed by Haar measure on $O(N)$, which we again must note, is computationally prohibitive in most settings.

We have presented theory and numerics to show how the random butterfly matrices sit in relation to other classes of random orthogonal/unitary matrices. And while there is no obvious link between the established theoretical properties of Haar-butterfly matrices and their performance in practice, the numerical results are quite suggestive.  Further work will be in the direction of making this link more concrete.

 \begin{table} \centerline{\begin{tabular}{c|c|c}
                Distribution on $U\pp{N}$ & $ \mathrm{Var}\left( \tr (U^{(N)})^k \right) $ & Convergence of $\frac{\tr (U^{(N)})^k - \mathbb E\left[\tr \left( U^{(N)} \right)^k \right]}{\sqrt{\mathrm{Var}\left( \tr (U^{(N)})^k \right)}}$\\
                && \\
                \hline
                && \\
                Haar measure on $U(N)$ & $k$  & \begin{tabular}{c} in distribution\\
                                                                            to complex standard normal\end{tabular}  \\
                && \\
                \hline
                && \\
                Haar-butterfly distribution &  $N$  & almost surely to zero\\
                && \\
                \hline
                && \\
                iid eigenvalues\footnote{By \emph{iid eigenvalues} we mean that the eigenvalues are given by $\lambda_j = \E^{\I \theta_j}$ where $\theta_j$ are iid uniform on $[0, 2\pi)$.} & $N$ & \begin{tabular}{c} in distribution\\
                                                                            to complex standard normal \end{tabular} 
                           \end{tabular}}
                         \vspace{.1in}
\caption{A summary of the statistical behavior of Haar-butterfly matrices compared to other well-known distributions on orthogonal/unitary matrices.  Haar-butterfly matrices share the same variance as those with iid while only requiring $\log_2 N$ random variables in their construction. \label{table} }
                         \end{table}

\bibliography{library}
\end{document}

%% file: listing.tex
\lstdefinelanguage{Julia}%
  {morekeywords={abstract,break,case,catch,const,continue,do,else,elseif,%
      end,export,false,for,function,immutable,import,importall,if,julia,in,%
      macro,module,otherwise,quote,return,switch,true,try,type,typealias,%
      using,while},%
   sensitive=true,%
   alsoother={$},%
   morecomment=[l]\#,%
   morecomment=[n]{\#=}{=\#},%
   morestring=[s]{"}{"},%
   morestring=[m]{'}{'},%
}[keywords,comments,strings]%

%% file: mcom-l-template.bbl
\providecommand{\bysame}{\leavevmode\hbox to3em{\hrulefill}\thinspace}
\providecommand{\MR}{\relax\ifhmode\unskip\space\fi MR }
\providecommand{\MRhref}[2]{%
  \href{http://www.ams.org/mathscinet-getitem?mr=#1}{#2}
}
\providecommand{\href}[2]{#2}
\begin{thebibliography}{10}

\bibitem{Avron2010}
H~Avron, P~Maymounkov, and S~Toledo, \emph{{Blendenpik: Supercharging LAPACK's
  Least-Squares Solver}}, SIAM Journal on Scientific Computing \textbf{32}
  (2010), no.~3, 1217--1236.

\bibitem{Boutsidis2013}
C~Boutsidis and A~Gittens, \emph{{Improved Matrix Algorithms via the Subsampled
  Randomized Hadamard Transform}}, SIAM Journal on Matrix Analysis and
  Applications \textbf{34} (2013), no.~3, 1301--1340.

\bibitem{Diaconis2001}
P~Diaconis and S~N Evans, \emph{{Linear functionals of eigenvalues of random
  matrices}}, Transactions of the American Mathematical Society \textbf{353}
  (2001), no.~07, 2615--2634.

\bibitem{Diaconis1994}
P~Diaconis and M~Shahshahani, \emph{{On the Eigenvalues of Random Matrices}},
  Journal of Applied Probability \textbf{31} (1994), no.~1994, 49.

\bibitem{Forrester2010}
P~Forrester, \emph{{Log-gases and random matrices}}, Princeton University
  Press, 2010.

\bibitem{Halko2011}
N~Halko, P~G Martinsson, and J~A Tropp, \emph{{Finding Structure with
  Randomness: Probabilistic Algorithms for Constructing Approximate Matrix
  Decompositions}}, SIAM Review \textbf{53} (2011), no.~2, 217--288.

\bibitem{Ipsen2014}
I~C~F Ipsen and T~Wentworth, \emph{{The Effect of Coherence on Sampling from
  Matrices with Orthonormal Columns, and Preconditioned Least Squares
  Problems}}, SIAM Journal on Matrix Analysis and Applications \textbf{35}
  (2014), no.~4, 1490--1520.

\bibitem{Kac1959}
M~Kac, \emph{{Statistical Independence in Probability, Analysis and Number
  Theory}}, Mathematical Association of America, Washington D.C., 1959.

\bibitem{Lerman2014}
G~Lerman and T~Maunu, \emph{{Fast algorithm for robust subspace recovery}},
  (2014).

\bibitem{Liberty2007}
E~Liberty, F~Woolfe, P-G Martinsson, V~Rokhlin, and M~Tygert, \emph{{Randomized
  algorithms for the low-rank approximation of matrices.}}, Proceedings of the
  National Academy of Sciences of the United States of America \textbf{104}
  (2007), no.~51, 20167--72.

\bibitem{Mezzadri2006}
F~Mezzadri, \emph{{How to generate random matrices from the classical compact
  groups}}, Notices of the AMS \textbf{54} (2006), 592--604.

\bibitem{Nachbin1976}
L~Nachbin, \emph{{The Haar integral}}, R.E. Krieger Pub. Co, 1976.

\bibitem{Parker1995}
D~S Parker, \emph{{Random Butterfly Transformations with Applications in
  Computational Linear Algebra}}, Tech. report, UCLA, 1995.

\bibitem{Stewart1980}
G~W Stewart, \emph{{The efficient generation of random orthogonal matrices with
  an application to condition estimators}}, SIAM Journal on Numerical Analysis
  (1980).

\bibitem{Tropp2011}
J~A Tropp, \emph{{Improved analysis of the subsampled randomized Hadamard
  transform}}, Advances in Adaptive Data Analysis \textbf{03} (2011),
  no.~01n02, 115--126.

\bibitem{VanLoan1992}
Charles {Van Loan}, \emph{{Computational Frameworks for the Fast Fourier
  Transform}}, Society for Industrial and Applied Mathematics, jan 1992.

\bibitem{Weil1951}
A~Weil, \emph{{L'int{\'{e}}gration dans les groupes topologiques et ses
  applications}}, Paris, Hermann, Paris, 1951.

\end{thebibliography}
